\documentclass{article}%
\usepackage{amsfonts}
\usepackage{amsmath}
\usepackage{amssymb}
\usepackage{graphicx}%
\providecommand{\U}[1]{\protect\rule{.1in}{.1in}}
\newtheorem{theorem}{Theorem}[section]

\newtheorem{corollary}[theorem]{Corollary}

\newtheorem{definition}[theorem]{Definition}
\newtheorem{example}[theorem]{Example}

\newtheorem{remark}[theorem]{Remark}

\newenvironment{proof}[1][Proof]{\noindent\textbf{#1.} }{\ \rule{0.5em}{0.5em}}
\begin{document}

\title{Topological equivalences for one-parameter bifurcations of maps}
\author{Balibrea\thanks{Department of Mathematics, University of Murcia, Campus de
Espinardo, 30100 Murcia, Spain}, Francisco, Oliveira\thanks{Center for
Mathematical Analysis Geometry and Dynamical Systems, Department of
Mathematics, Instituto Superior T\'{e}cnico, Universidade de Lisboa, Av.
Rovisco Pais, 1, 1049-001 Lisboa, Portugal}, Henrique M. and
Valverde\thanks{Department of Mathematics, University of Castilla-La Mancha,
2071-Albacete, Spain}, Jose C.}
\maketitle

\begin{abstract}
\noindent Homeomorphisms allowing us to prove topological equivalences between
one-parameter families of maps undergoing the same bifurcation are constructed
in this paper. This provides a solution for a classical problem in bifurcation
theory that was set out three decades ago and remained unexpectedly
unpublished until now.\bigskip



\end{abstract}


\section{Introduction}

One of the fundamental problems in the study of nonlinear dynamical systems is
to know if the behavior of a system changes under small perturbations. Roughly
speaking, when the dynamics of a system changes, it is said that a
\emph{bifurcation} occurs. On the contrary, if no change happens, it is said
that the system is \emph{structurally stable}.

Usually, perturbations of a dynamical system are associated to variations of
parameters involved in the equation(s) which define such a system (see
\cite{Guc77} for the original definition). Since any value of the parameters
defines a particular dynamical system, if we consider all together, we have a
parametric family of them. When studying bifurcations of parametric families
of dynamical systems, two techniques allow us to simplify this study, namely,
the \emph{center manifold theory} and the \emph{normal forms} method. The
first one provides a reduction of the dimensionality (see for example
\cite{Car81}), while the second one provides a simplification of the
nonlinearity, providing prototypes of behavior for vast classes of nonlinear
systems. The method of the normal forms began with Poincar\'{e} \cite{Poi29}
and was developed by Arnold (e.g. see \cite{Arn83}, \cite{AAIS94} or
\cite{GuHo83}). In a few words, this method consists in constructing a simple
polynomial family, named \emph{normal form}, using the conditions that provide
a bifurcation, and demonstrating that any family satisfying those conditions
is topologically equivalent to this simpler polynomial family. Note that, in
such a case, the dynamical behavior of any of these families can be inferred
from the dynamical behavior of the normal form which is obviously simpler to
be studied.

In this work, we consider local one-parameter bifurcations for maps.
Hartman-Grobman's theorem establishes that to study local bifurcations in
parametric families, it suffices to consider those parameter values for which
the corresponding map presents a non-hyperbolic fixed point, that is, the
Jacobian matrix evaluated at this fixed point has eigenvalues of modulus 1. In
particular, the simplest cases appear when only an eigenvalue has modulus 1,
i.e., 1 or -1. When the eigenvalue is equal to 1, three kinds of local
bifurcations can appear called \emph{fold, transcritical} and \emph{pitchfork}%
; while when the eigenvalue is equal to -1 a bifurcation named \emph{flip} or
\emph{period doubling} can appear. The first three ones imply the change in
the number and stability of the fixed points; while the fourth one involves
additionally the appearance of a 2-periodic orbit (see for instance
\cite{Wig90}).

In \cite{Wig90}, it is shown that under some nonzero conditions up to the
third order of the derivatives, these local bifurcations of one-parameter
families of maps (fold, transcritical, pitchfork, flip) appear. In
\cite{BaVa99}, similar results are obtained for higher order nonzero
conditions, generalizing the necessary conditions for the appearance of such bifurcations.

Although results in \cite{BaVa99, Wig90} prove that the same number of fixed
points (or fixed points and period-2 points in the flip case) appears with
identical type of stability, the problem of proving the topological
equivalence between any family verifying the same bifurcation conditions and
the corresponding simplest normal form posed originally in \cite{Arn83}
remained unpublished until now, as claimed in \cite{Kuz95}.

Actually, as done in \cite{Arn83} and \cite{Kuz95} for the fold and flip
bifurcations under the lowest order conditions and in \cite{BaVa00} for the
fold, transcritical, pitchfork and flip bifurcations under higher order
conditions, after applying successive diffeomorphic changes of variables and
re-scalings, it is possible to demonstrate the topological equivalence between
a family satisfying the bifurcation conditions of a specific order and the non
truncated normal form of this same order. With non truncated normal form we
refer to the normal form with higher-order terms of the Taylor polynomial that
are not fixed by the bifurcation conditions. Nevertheless, that the truncation
of higher order terms does not affect the topological type had not been proved
yet. This is not a minor question since in cases like the Neimark-Sacker
bifurcation (see for instance \cite{BaVa99}) this truncation does affect its
topological type.

The present article gives a solution for this classical problem in bifurcation
theory, set out more than three decades ago in \cite{Arn83}. As a matter of
fact, we provide a complete proof of topological equivalences between any two
families of maps undergoing one of the above mentioned bifurcations in a more
general way than originally posed, as shown in the results and examples herein.

In this sense, our main result consists in the construction of a homeomorphism
$h$ of conjugation between any two order preserving real homeomorphisms $f$
and $g$ in suitable real open intervals $J$ and $I$ both with the same number
of fixed points with the same stabilities, i.e., a homeomorphism
$h:J\rightarrow I$, such that $h\circ f=g\circ h$. When the two homeomorphisms
$f$ and $g$ are order reversing, we construct a similar conjugacy in a
suitable neighborhood of the unique fixed point.

Despite the mathematical analysis focus on topological conjugations, this
construction, applied to the case of one-parameter families undergoing the
same kind of bifurcation, allows us to provide a solution for the mentioned
classical problem in bifurcation theory.



Of course, although the results are given for bifurcations of fixed points,
they also apply to periodic points when the relevant iterate of the map is
considered. For simplicity, we have stated all of our results in the context
of one-dimensional maps. However, they can also be used in order to draw these
bifurcations in one-parameter families of $n$-dimensional maps or, even more
generally, of Banach spaces.

The organization of this paper is as follows. In Section \ref{Section2}, we
construct conjugacies among any two homeomorphisms having the same number and
stability of fixed points. As a result, in Section \ref{Section3}, we show how
these results in the previous section allow us to solve the problem of
demonstrating the topological equivalence between any two one-parameter
families undergoing the same type of bifurcation, even in a more general way
than that in which it was posed originally. Finally, in Section \ref{Section4}
we present conclusions and future research directions.

\section{\label{Section2}Technique for the construction of topological
conjugacies}

In this section, we develop a technique to construct homeomorphisms which give
us a conjugacy between two homeomorphisms $g$ and $f$ satisfying similar
conditions concerning the number and stability of their fixed points. These
homeomorphisms will allows us to demonstrate the topological equivalence
between any family satisfying certain bifurcations conditions and the
corresponding normal form in Section \ref{Section3}.

Along this paper, we adopt definitions in \cite{Kuz95} concerning topological
conjugacy and topological equivalence. When nothing else stated, the letter
$n$ is reserved for natural numbers and Greek letters for real parameters.

In this article, we consider continuous maps with at most a countable number
of isolated fixed points, i.e., without accumulation points and separated by
non-empty intervals. Results for non-countable fixed points or with
accumulation points are not in the scope of this work. This avoids maps that
are coincident with the identity map in some non-empty interval and
topological neutral fixed points.

One important concept in this work is the following, that one can find in
\cite{Elaydi}.

\begin{definition}
A fixed point $x_{F}$ of a map $g$ is said to be \emph{semi-attracting from
the left} (resp. \emph{from the right}), if there exists $\eta>0$ such that
for $x\in\left(  x_{F}-\eta,x_{F}\right)  $ (resp. $x\in\left(  x_{F}%
,x_{F}+\eta\right)  $), $\lim_{n\longrightarrow+\infty}g^{n}\left(  x\right)
=x_{F}$.
\end{definition}

In view of this concept, the notion of a semi-repelling fixed point to the
left (resp. to the right) is clear. Actually, for a fixed point $x_{F}$ of an
increasing homeomorphism $g$, it is \emph{semi-repelling to the left} (resp.
\emph{to the right}), iff the same fixed point $x_{F}$ is semi-attracting from
the left (resp. right) for the inverse $g^{-1}$.

Notice that an attracting fixed point is semi-attracting from the left and
simultaneously from the right, and the same happens for repelling fixed points
that must be simultaneously semi-repelling to the left and to the right. In
particular, semi-attracting fixed points of decreasing homeomorphisms are
always attracting.

We will also use the term left semi-attracting (resp. left semi-repelling) and
right semi-attracting (resp. right semi-repelling) to simplify the
nomenclature of these concepts.

\begin{definition}
We shall call \emph{transverse\footnote{Relative to the diagonal, i.e., the
identity map.} fixed} point to any fixed one being attracting or repelling.
\end{definition}

For differentiable maps, every hyperbolic fixed point is transverse. But, the
reciprocal is not true since, for instance, the real map $x+x^{3}$ has the
non-hyperbolic transverse fixed point $0$.

\begin{definition}
We shall call \emph{mixed-stability fixed point} to any fixed point being left
semi-attracting and right semi-repelling or left semi-repelling and right semi-attracting.
\end{definition}

The above notions can be more refined, but in our work we only need the above
concepts which are related to bifurcation theory in metric spaces.

\subsection{Maps with one fixed point}

Here, we study the situation of only one fixed point for each map $g$ and
$f$.
The left and right of the fixed points are treated separately, this treatment
allows us to study maps with more than one fixed point in subsection
\ref{Main}.

\begin{theorem}
\label{XRStable} Let $g:I\longrightarrow g(I)$ and $f:J\longrightarrow f(J)$
two increasing homeomorphisms with domains $I=[b,x_{F}]$ and $J=[a,\overline
{x}_{F}]$ respectively. Suppose that they verify the following conditions:


\begin{enumerate}
\item $g\left(  x\right)  >x$, for all $x\in\left[  b,x_{F}\right)  $.

\item $f\left(  x\right)  >x$, for all $x\in\left[  a,\overline{x}_{F}\right)
$.

\item $x_{F}$ is the (unique) fixed point of $g$ and $\overline{x}_{F}$ is the
(unique) fixed point of $f$.
\end{enumerate}

Then, there exists a topological conjugacy between $g$ and $f$, i.e.,
a homeomorphism $h:J\longrightarrow I.$, not necessarily unique, such that%
\begin{equation}
g\circ h\left(  x\right)  =h\circ f\left(  x\right)  . \label{COnjug}%
\end{equation}

\end{theorem}

\begin{proof}
First of all, we note that both $x_{F}$ and $\overline{x}_{F}$ are
semi-attracting from the left. Observe that, thanks to hypotheses 1 and 2,
$f(J)\subseteq J$ and $g(I)\subseteq I$. In fact, $f(J)=[f(a),\overline{x}%
_{F}]\subseteq\left[  a,\overline{x}_{F}\right]  =J$ and $g(I)=\left[
g(b),x_{F}\right]  \subseteq\left[  b,x_{F}\right]  =I$, since $f(a)>a$ and
$g(b)>b$ by the hypotheses. In view of this, the homeomorphism to be defined
will have the domain $J$ and the image $I$ and in this way it can be
restricted to $f(J)$ and $g(I)$

Consider $a$, the leftmost point of $J$.
We will construct the homeomorphism $h$ subject to the condition $h\left(
a\right)  =b$. In fact, this arbitrary choice allows us to deduce that there
exist infinitely many different topological conjugacies for the same functions
$g$ and $f$. Now, we can consider (since there are infinity of them) an
increasing homeomorphism $h_{0}$ in the domain $D_{0}=\left[  a,f\left(
a\right)  \right]  $ (a \emph{fundamental domain} in \cite{Melo})%
\[
h_{0}:\left[  a,f\left(  a\right)  \right]  \longrightarrow\left[  b,g\left(
b\right)  \right]  ,
\]
such that $h_{0}\left(  a\right)  =b$ and $h_{0}\left(  f\left(  a\right)
\right)  =g\left(  b\right)  $\footnote{For instance,
\[
h_{0}\left(  x\right)  =b+\left(  x-a\right)  \frac{g\left(  b\right)
-b}{f\left(  a\right)  -a},\text{ }x\in D_{0}%
\]
works perfectly to start the process. Any other continuous map with the same
properties works just fine.}.

Note that we have chosen arbitrarily the homeomorphism $h_{0}$ subject to the
condition
\[
g\left(  h_{0}\left(  a\right)  \right)  =h_{0}\left(  f\left(  a\right)
\right)  .
\]
We now use $h_{0}$ to construct a topological conjugacy between $g$ and $f$.
The main ingredient of the proof is the non-locality of this construction
process.

We have started the construction of a topological equivalence by the
restriction of $h$ to the fundamental domain which is $h_{0}$. Consider the
points $x$ in the fundamental domain $D_{0}=\left[  a,f\left(  a\right)
\right]  $, the right side of the conjugacy equation (\ref{COnjug}), i.e.,
$h\circ f\left(  x\right)  $, acts on $D_{0}$. Obviously $f\left(
D_{0}\right)  =\left[  f\left(  a\right)  ,f^{2}\left(  a\right)  \right]
=D_{1}$. In order to compute directly $h\left(  x\right)  $ when $x\in D_{1}$,
we use the left hand side of the conjugacy equation and define $h_{1}\left(
x\right)  $ when $x\in D_{1}$, i.e., the restriction of $h$ to the interval
$D_{1}$. The left side of the conjugacy equation when $x\in D_{0}$ is well
defined, it is $g\circ h_{0}\left(  x\right)  $. We obtain a definition of
$h_{1}\left(  x\right)  $, i.e., the restriction of $h\left(  x\right)  $ to
the interval $D_{1}$, forcing the diagram to be commutative, that is,%
\[
h_{1}\circ f\left(  x\right)  =g\circ h_{0}\left(  x\right)  \text{, }x\in
D_{0}%
\]
or%
\[
h_{1}\left(  x\right)  =\left(  g\circ h_{0}\circ f^{-1}\right)  \left(
x\right)  \text{, }x\in D_{1}\text{.}%
\]
In such a way, we can define the sequence of intervals%
\[
D_{n}=\left[  f^{n}\left(  a\right)  ,f^{n+1}\left(  a\right)  \right]
=f^{n}\left(  D_{0}\right)  \text{,}%
\]
where $f^{n}$ stands for the $n$-th composition of $f$ with itself, and it is
possible to extend the definition of successive restrictions $h_{n}$ of $h$ to
the intervals $D_{n}$, using the same procedure. Thus, for $n=2$, we get%
\[
h_{2}\left(  x\right)  =\left(  g\circ h_{1}\circ f^{-1}\right)  \left(
x\right)  \text{, }x\in D_{2}%
\]
or%
\[
h_{2}\left(  x\right)  =\left(  g^{2}\circ h_{0}\circ f^{-2}\right)  \left(
x\right)  \text{, }x\in D_{2}\text{,}%
\]
where $f^{-2}$ stands for $f^{-1}\circ f^{-1}$. In general%
\[
h_{n}\left(  x\right)  =\left(  g^{n}\circ h_{0}\circ f^{-n}\right)  \left(
x\right)  \text{, }x\in D_{n}\text{.}%
\]
This construction is well done. Effectively, let $x\in D_{j}$ with $0<j<n$,
then
\begin{align*}
g\circ h_{j}\left(  x\right)   &  =h_{j+1}\circ f\left(  x\right) \\
g\circ(g^{j}\circ h_{0}\circ f^{-j})\left(  x\right)   &  =\left(
g^{j+1}\circ h_{0}\circ f^{-j-1}\right)  \circ f\left(  x\right) \\
g^{j+1}\circ h_{0}\circ f^{-j}\left(  x\right)   &  =g^{j+1}\circ h_{0}\circ
f^{-j}\left(  x\right)  \text{.}%
\end{align*}
Consider now the image domains%
\[
D_{n}^{\prime}=\left[  g^{n}\left(  b\right)  ,g^{n+1}\left(  b\right)
\right]
\]
and the homeomorphisms
\[
h_{n}: D_{n}\longrightarrow D_{n}^{\prime},\quad\text{ }h_{n}(x)=\left(
g^{n}\circ h_{0}\circ f^{-n}\right)  (x),
\]
which are increasing in each of the domains, since they are the composition of
increasing functions.

At this point, observe that the union of the intervals $D_{n}=\left[
f^{n}\left(  a\right)  ,f^{n+1}\left(  a\right)  \right]  $ is%
\[%
{\displaystyle\bigcup\limits_{n=0}^{\infty}}
D_{n}=\left[  a,\overline{x}_{F}\right)  ,
\]
while the union of the intervals $D_{n}^{\prime}=\left[  g^{n}\left(
b\right)  ,g^{n+1}\left(  b\right)  \right]  $ is%
\[%
{\displaystyle\bigcup\limits_{n=0}^{\infty}}
D_{n}^{\prime}=\left[  b,x_{F}\right)  ,
\]
This is true, since the sequence $(f^{n}(a))_{n\in\mathbb{N}}$ is increasing
and bounded by $\overline{x}_{F}$. Therefore, it has a limit, namely $L$. But,
due to the continuity of the $f$,
\[
f(L)=f\left(  \lim_{n\rightarrow\infty}f^{n}(a)\right)  =\lim_{n\rightarrow
\infty}f^{n+1}(a)=L
\]
Hence, the limit of the sequence is a fixed point of $f$ and, by hypothesis 4,
the unique fixed point in the interval $[a,\overline{x}_{F}]$ is precisely
$\overline{x}_{F}$. A similar reasoning can be used for the case of $g$ and
the union of the intervals $D_{n}^{\prime}$.

Then, we can consider the piecewise function $h^{\circ}$ defined as
\[
h^{\circ}\left(  x\right)  =\left\{
\begin{array}
[c]{ccccc}%
h_{0}\left(  x\right)  & \text{:} & \left[  a,f\left(  a\right)  \right]  &
\rightarrow & \left[  b,g\left(  b\right)  \right]  \bigskip\\
h_{1}\left(  x\right)  & \text{:} & \left[  f\left(  a\right)  ,f^{2}\left(
a\right)  \right]  & \rightarrow & \left[  g\left(  b\right)  ,g^{2}\left(
b\right)  \right]  \bigskip\\
\cdots &  &  &  & \bigskip\\
h_{n}\left(  x\right)  & \text{:} & \left[  f^{n}\left(  a\right)
,f^{n+1}\left(  a\right)  \right]  & \rightarrow & \left[  g^{n}\left(
b\right)  ,g^{n+1}\left(  b\right)  \right]  \bigskip\\
\cdots &  &  &  & .
\end{array}
\right.
\]
The function $h^{\circ}$ is a (increasing) homeomorphism piecewise defined in
$\left[  a,\overline{x}_{F}\right)  $ which is constituted by increasing
homeomorphisms joined together. Finally, we extend $h^{\circ}$ by continuity
to a new homeomorphism, our desired $h$, such that%
\begin{align*}
h\left(  x\right)   &  =h^{\circ}\left(  x\right)  \text{, for }x\in\left[
a,\overline{x}_{F}\right) \\
h\left(  \overline{x}_{F}\right)   &  =x_{F}\text{.}%
\end{align*}
The homeomorphism $h$ defined in $\left[  a,\overline{x}_{F}\right]  $ is the
required conjugation.
\end{proof}

\vspace*{0.5cm}

\begin{corollary}
\label{XLStable}Let $g:I\longrightarrow g(I)$ and $f:J\longrightarrow f(J)$
two increasing homeomorphisms with domains $I=\left[  x_{F},b\right]  $ and
$J=\left[  \overline{x}_{F},a\right]  $. Suppose that they verify the
following conditions:

\begin{enumerate}
\item $g\left(  x\right)  <x$, for all $x\in\left(  x_{F},b\right]  $.

\item $f\left(  x\right)  <x$, for all $x\in\left(  \overline{x}_{F},a\right]
$.

\item $x_{F}$ is the (unique) fixed point of $g$ and $\overline{x}_{F}$ is the
(unique) fixed point of $f$.
\end{enumerate}

Then, there exists a topological conjugacy between $g$ and $f$, i.e.,
a homeomorphism $h:J\longrightarrow I$, not necessarily unique, such that%
\[
g\circ h\left(  x\right)  =h\circ f\left(  x\right)  .
\]

\end{corollary}

\begin{proof}
We note that both $x_{F}$ and $\overline{x}_{F}$ are semi-attracting from the
right. The proof is similar to the one of Theorem \ref{XRStable}. The unique
difference is that, in this case, due to the hypotheses 1 and 2, now the
domains are
\[
D_{n}=\left[  f^{n+1}\left(  a\right)  ,f^{n}\left(  a\right)  \right]  ,
\]%
\[
D_{n}^{\prime}=\left[  g^{n+1}\left(  b\right)  ,g^{n}\left(  b\right)
\right]
\]
and the homeomorphisms are
\[
h_{n}:D_{n}\longrightarrow D_{n}^{\prime},\quad\text{ }h_{n}\left(  x\right)
=\left(  g^{n}\circ h_{0}\circ f^{-n}\right)  \left(  x\right)
\]

\end{proof}

\vspace*{0.5cm}

\begin{corollary}
\label{XRuns} Let $g:I\longrightarrow g(I)$ and $f:J\longrightarrow f(J)$ two
increasing homeomorphisms with domains $I=\left[  b,x_{F}\right]  $ and
$J=\left[  a,\overline{x}_{F}\right]  $. Suppose that they verify the
following conditions:

\begin{enumerate}
\item $g\left(  x\right)  <x$, for all $x\in\left[  b,x_{F}\right)  $.

\item $f\left(  x\right)  <x$, for all $x\in\left[  a,\overline{x}_{F}\right)
$.

\item $x_{F}$ is the (unique) fixed point of $g$ and $\overline{x}_{F}$ is the
(unique) fixed point of $f$.
\end{enumerate}

Then, there exists a topological conjugacy between $g$ and $f$, i.e.,
a homeomorphism $h:J\longrightarrow I$, not necessarily unique, such that%
\[
g\circ h\left(  x\right)  =h\circ f\left(  x\right)  .
\]

\end{corollary}

\begin{proof}
First of all, we note that both $x_{F}$ and $\overline{x}_{F}$ are both
semi-repelling to the left. The proof is similar to the one of
Theorem \ref{XRStable}. Consider $a$, the leftmost point of $J$. Since
$f\left(  x\right)  <x$ we have $f^{-1}\left(  x\right)  >x$, so
$f^{-1}\left(  a\right)  >a$. The same happens to $g$, so $g^{-1}\left(
b\right)  >b$. We note that $f^{-1}$ and $g^{-1}$ have the same fixed points
of $f$ and $g$, which are now attracting. The sequences $f^{-n}\left(
a\right)  $ and $g^{-n}\left(  x\right)  $ are increasing and bounded by the
fixed points $\overline{x}_{F}$ and $x_{F}$.

At this point, we can apply Theorem \ref{XRStable} to $f^{-1}$ and $g^{-1}$.
Therefore, $f^{-1}$ and $g^{-1}$ are topologically equivalent%
\[
g^{-1}=h\circ f^{-1}\circ h^{-1} ,
\]
with $h: J \longrightarrow I\text{.}$ Thus inverting the maps in the last
equality, we get
\[
g=h\circ f\circ h^{-1},
\]
as desired.
\end{proof}

\vspace*{0.5cm}

\begin{corollary}
\label{XLuns} Let $g:I\longrightarrow g(I)$ and $f:J\longrightarrow f(J)$ two
increasing homeomorphisms with domains $I=\left[  x_{F},b\right]  $ and
$J=\left[  \overline{x}_{F},a\right]  $. Suppose that they verify the
following conditions

\begin{enumerate}
\item $g\left(  x\right)  >x$, for all $x\in\left(  x_{F},b\right]  $.

\item $f\left(  x\right)  >x$, for all $x\in\left(  \overline{x}_{F},a\right]
$.

\item $x_{F}$ is the (unique) fixed point of $g$ and $\overline{x}_{F}$ is the
(unique) fixed point of $f$.
\end{enumerate}

Then, there exists a topological conjugacy between $g$ and $f$, i.e.,
a homeomorphism $h:J\longrightarrow I$, not necessarily unique, such that%
\[
g\circ h\left(  x\right)  =h\circ f\left(  x\right)  .
\]

\end{corollary}

\begin{proof}
First of all we note that both $x_{F}$ and $\overline{x}_{F}$ are
semi-repelling to the right. The proof is similar to the ones in the previous results.
\end{proof}

\vspace*{0.5cm}



\subsection{Maps with consecutive pairs of fixed points}

Consider two homeomorphisms $f$ and $g$ with a finite number of isolated fixed
points. In this subsection, we study the construction of a homeomorphism $h$
in intervals between two fixed points of $g$ and $f$. Of course, the obtained
results can be applied between any two consecutive pairs of fixed points of
$g$ and $f$. The exterior of the union of such intervals between two
consecutive fixed points can be treated using the results of the previous subsection.

\begin{theorem}
\label{T1}Let $g:I\longrightarrow I$ and $f:J\longrightarrow J$ two increasing
homeomorphisms defined in the intervals $I=\left[  x_{L},x_{R}\right]  $ and
$J=\left[  \overline{x}_{L},\overline{x}_{R}\right]  $. Suppose that they
verify the following conditions:

\begin{enumerate}

\item $g\left(  x\right)  >x$, for all $x\in I=\left(  x_{L},x_{R}\right)  $.

\item $f\left(  x\right)  >x$, for all $x\in J=\left(  \overline{x}%
_{L},\overline{x}_{R}\right)  $.

\item The endpoints of the intervals $I, J$ are the unique fixed points of $g$
and $f$ respectively.


\end{enumerate}

Then, there exists a topological conjugacy between $g$ and $f$, i.e.,
a homeomorphism $h:J\longrightarrow I$, not necessarily unique, such that%
\begin{equation}
g\circ h\left(  x\right)  =h\circ f\left(  x\right)  \label{conjug}%
\end{equation}

\end{theorem}

\begin{proof}
First of all, it is easy to check that $x_{L}$ and $\overline{x}_{L}$ are
semi-repelling fixed points to the right while $x_{R}$ and $\overline{x}_{R}$
are semi-attracting fixed points from the left.

To prove the result, we fix one image of $h$ at one particular point%
\[
h\left(  a\right)  =b\text{,}%
\]
where $\overline{x}_{L}<a<\overline{x}_{R}$ and $x_{L}<b<x_{R}$.

Now, we can consider an increasing homeomorphism $h_{0}$ in the domain
$D_{0}=\left[  a,f\left(  a\right)  \right]  $
\[
h_{0}:\left[  a,f\left(  a\right)  \right]  \longrightarrow\left[  b,g\left(
b\right)  \right]  ,
\]
such that $h_{0}\left(  a\right)  =b$ and $h_{0}\left(  f\left(  a\right)
\right)  =g\left(  b\right)  $. Note that we have chosen arbitrarily the
homeomorphism $h_{0}$ subject to the condition $g\left(  h_{0}\left(
a\right)  \right)  =h_{0}\left(  f\left(  a\right)  \right)  $.



We use $h_{0}$ to construct a topological conjugacy between $g$ and $f$. The
procedure is divided in two steps: in the first one, we construct the
conjugacy at forward intervals, i.e., right of the fundamental domain; in the
second one, we construct the conjugacy at backward intervals, i.e., left of
the fundamental domain.

We have started the construction of a topological equivalence by the
restriction of $h$ to the fundamental domain which is $h_{0}$. Consider the
points $x$ in the fundamental domain $D_{0}=\left[  a,f\left(  a\right)
\right]  $, the right side of the conjugacy equation (\ref{conjug}), i.e.,
$h\circ f\left(  x\right)  $, acts on $D_{0}$. Obviously $f\left(
D_{0}\right)  =\left[  f\left(  a\right)  ,f^{2}\left(  a\right)  \right]
=D_{1}$. In order to compute directly $h\left(  x\right)  $ when $x\in D_{1}$,
we use the left hand side of the conjugacy equation and define $h_{1}\left(
x\right)  $ when $x\in D_{1}$, i.e., the restriction of $h$ to the interval
$D_{1}$. The left side of the conjugacy equation when $x\in D_{0}$ is well
defined, it is $g\circ h_{0}\left(  x\right)  $. We obtain a definition of
$h_{1}\left(  x\right)  $, i.e., the restriction of $h\left(  x\right)  $ to
the interval $D_{1}$, forcing the diagram to be commutative, that is,
\[
h_{1}\circ f\left(  x\right)  =g\circ h_{0}\left(  x\right)  \text{, }x\in
D_{0}%
\]
or%
\[
h_{1}\left(  x\right)  =\left(  g\circ h_{0}\circ f^{-1}\right)  \left(
x\right)  \text{, }x\in D_{1}\text{.}%
\]
In such a way, we can define the sequence of intervals%
\[
D_{n}=\left[  f^{n}\left(  a\right)  ,f^{n+1}\left(  a\right)  \right]
=f^{n}\left(  D_{0}\right)  \text{,}%
\]
where $f^{n}$ stands for the $n$-th composition of $f$ with itself, and it is
possible to extend the definition of successive restrictions $h_{n}$ of $h$ to
the intervals $D_{n}$, using the same procedure. In general%
\[
h_{n}\left(  x\right)  =\left(  g^{n}\circ h_{0}\circ f^{-n}\right)  \left(
x\right)  \text{, }x\in D_{n}\text{.}%
\]
At this point, observe that the union of the intervals $D_{n}=\left[
f^{n}\left(  a\right)  ,f^{n+1}\left(  a\right)  \right]  $ is%
\[%
{\displaystyle\bigcup\limits_{n=0}^{\infty}}
D_{n}=\left[  a,\overline{x}_{R}\right)  ,
\]
This is true, since the sequence $(f^{n}(a))_{n\in\mathbb{N}}$ is increasing
and bounded by $\overline{x}_{R}$. Therefore, it has a limit, namely $L$. But,
due to the continuity of the $f$,
\[
f(L)=f\left(  \lim_{n\rightarrow\infty}f^{n}(a)\right)  =\lim_{n\rightarrow
\infty}f^{n+1}(a)=L
\]
Hence, the limit of the sequence is a fixed point of $f$ and, by hypothesis 4,
the unique fixed point in the interval $[a,\overline{x}_{R}]$ is precisely
$\overline{x}_{R}$.

Summarizing, we have the function $h_{\rightarrow}^{\circ}$ piecewise defined
as
\[
h_{\rightarrow}^{\circ}\left(  x\right)  =\left\{
\begin{array}
[c]{ccccc}%
h_{0}\left(  x\right)  & \text{:} & \left[  a,f\left(  a\right)  \right]  &
\rightarrow & \left[  b,g\left(  b\right)  \right]  \bigskip\\
h_{1}\left(  x\right)  & \text{:} & \left[  f\left(  a\right)  ,f^{2}\left(
a\right)  \right]  & \rightarrow & \left[  g\left(  b\right)  ,g^{2}\left(
b\right)  \right]  \bigskip\\
\cdots &  &  &  & \bigskip\\
h_{n}\left(  x\right)  & \text{:} & \left[  f^{n}\left(  a\right)
,f^{n+1}\left(  a\right)  \right]  & \rightarrow & \left[  g^{n}\left(
b\right)  ,g^{n+1}\left(  b\right)  \right]  \bigskip\\
\cdots &  &  &  & \bigskip.
\end{array}
\right.
\]
Of course, $h_{\rightarrow}^{\circ}$ is continuous and increasing in each
interval and agrees at each common extreme point of two successive intervals.
Thus, $h_{\rightarrow}^{\circ}$ is a homeomorphism. Finally, observe that, as
in the case of $f$, the sequence $(g^{n}(a))_{n\in\mathbb{N}}$ is increasing
and bounded by $x_{R}$. Therefore, it has a limit and due to the continuity of
the $g$ and the hypothesis 4, its limit is precisely $x_{R}$. As a
consequence, we can prolong $h_{\rightarrow}^{\circ}$ to $\overline{x}_{R}$ by
continuity defining the complete homeomorphism $h_{\rightarrow}$ in the
compact interval $\left[  a,\overline{x}_{R}\right]  $ making $h_{\rightarrow
}\left(  \overline{x}_{R}\right)  =x_{R}$ and $h_{\rightarrow}\left(
x\right)  =h_{\rightarrow}^{\circ}\left(  x\right)  $ for $x\in\left[
a,\overline{x}_{R}\right)  $, i.e.,%
\[%
\begin{array}
[c]{cccc}%
h_{\rightarrow}: & \left[  a,\overline{x}_{R}\right]  & \longrightarrow &
\left[  a,x_{R}\right]  \text{,}%
\end{array}
\]
Using a similar reasoning, we can construct the topological conjugacy to the
left of the fundamental domain using now $f^{-1\,}$ and $g^{-1}$.

The restriction of $h$ to the fundamental domain is again $h_{0}$. Consider
the points $x$ in the fundamental domain $D_{0}=\left[  a,f\left(  a\right)
\right]  $. Now, we need to define the restrictions $h_{-n}$ of $h$ to the
intervals $D_{-n}=\left[  f^{-n}\left(  a\right)  ,f^{-n+1}\left(  a\right)
\right]  $. Using the same type of arguments of the forward construction we
find that%
\[
h_{-n}\left(  x\right)  =g^{-n}\circ h_{0}\circ f^{n}\left(  x\right)  \text{,
}x\in D_{-n}\text{.}%
\]
Effectively, let $x\in D_{-j}$ with $-n<-j\leq0$, then
\begin{align*}
g\circ h_{-j}\left(  x\right)   &  =h_{-j+1}\circ f\left(  x\right) \\
g\circ(g^{-j}\circ h_{0}\circ f^{j})\left(  x\right)   &  =\left(
g^{-j+1}\circ h_{0}\circ f^{j-1}\right)  \circ f\left(  x\right) \\
g^{-j+1}\circ h_{0}\circ f^{j}\left(  x\right)   &  =g^{-j+1}\circ h_{0}\circ
f^{j}\left(  x\right)  \text{.}%
\end{align*}
The rest of the backward process is just the same as in the forward case,
since now $\overline{x}_{L}$ and $x_{L}$ are attracting fixed points for
$f^{-1}$ and $g^{-1}$.

\medskip

Summarizing again, we can define the piecewise homeomorphism $h_{\leftarrow
}^{\circ}$%
\[
h_{\leftarrow}^{\circ}\left(  x\right)  =\left\{
\begin{array}
[c]{ccccc}%
h_{-1}\left(  x\right)  & \text{:} & \left[  f^{-1}\left(  a\right)  ,a\right]
& \rightarrow & \left[  g^{-1}\left(  b\right)  ,b\right]  \bigskip\\
h_{-2}\left(  x\right)  & \text{:} & \left[  f^{-2}\left(  a\right)
,f^{-1}\left(  a\right)  \right]  & \rightarrow & \left[  g^{-2}\left(
b\right)  ,g^{-1}\left(  b\right)  \right]  \bigskip\\
\cdots &  &  &  & \bigskip\\
h_{-n}\left(  x\right)  & \text{:} & \left[  f^{-n}\left(  a\right)
,f^{-n+1}\left(  a\right)  \right]  & \rightarrow & \left[  g^{-n}\left(
b\right)  ,g^{-n+1}\left(  b\right)  \right]  \bigskip\\
\cdots &  &  &  & .
\end{array}
\right.
\]

The homeomorphism $h_{\leftarrow}^{\circ}\left(  x\right)  $ defined in
$\left(  \overline{x}_{L},a\right]  $ can be prolonged by continuity to
$\overline{x}_{L}$ the same way as $h_{\rightarrow}$.

Therefore, joining the backward homeomorphism $h_{\leftarrow}$ and the forward
one $h_{\rightarrow}$, we obtain the increasing homeomorphism $h$ at the whole
interval $J=\left[  \overline{x}_{L},\overline{x}_{R}\right]  $ and, as
desired, the image of $h$ is $\left[  x_{L},x_{R}\right]  $.
\end{proof}

\begin{remark}
Actually, Theorem \ref{T1} can be seen as a corollary of Theorem
\ref{XRStable} applied directly to the right fixed points (in the intervals
$\left[  a,\overline{x}_{R}\right]  $ for $f$ and $\left[  b,x_{R}\right]  $
for $g$) and indirectly, using Corollary \ref{XLuns}\ applied to the left
fixed points (in the intervals $\left[  \overline{x}_{F},a\right]  $ for $f$
and $\left[  x_{F},b\right]  $ for $g$).
\end{remark}

\vspace*{0.5cm}

Naturally, a similar reasoning gives the same result when both semi-attracting
fixed points are on the left and the semi-repelling fixed points on the right.
Specifically, one can check the following corollary.

\begin{corollary}
\label{Cor1}Let $g:I\longrightarrow I$ and $f:J\longrightarrow J$ two
increasing homeomorphisms defined in the intervals $I=\left[  x_{L}%
,x_{R}\right]  $ and $J=\left[  \overline{x}_{L},\overline{x}_{R}\right]  $.
Suppose that they verify the following conditions:

\begin{enumerate}

\item $g\left(  x\right)  <x$, for all $x\in I=\left(  x_{L},x_{R}\right)  $.

\item $f\left(  x\right)  <x$, for all $x\in J=\left(  \overline{x}%
_{L},\overline{x}_{R}\right)  $.

\item The endpoints of the intervals $I, J$ are the unique fixed points of $g$
and $f$ respectively.


\end{enumerate}

Then, there exists a topological conjugacy between $g$ and $f$, i.e.,
a homeomorphism $h:J\longrightarrow I$, not necessarily unique, such that%
\begin{equation}
g\circ h\left(  x\right)  =h\circ f\left(  x\right)  . \label{conjug3}%
\end{equation}

\end{corollary}

\begin{proof}
Although in this case $x_{L}$ and $\overline{x}_{L}$ are semi-attracting fixed
points from the right while $x_{R}$ and $\overline{x}_{R}$ are semi-repelling
fixed points to the left, the proof is similar to the one in Theorem \ref{T1}.
\end{proof}

\vspace*{0.5cm}

Actually, the position of semi-attracting and semi-repelling fixed points is
not relevant. The semi-attracting and semi-repelling fixed points can be
opposite for $g$ or $f$. The original map $g$ is still conjugated to $f$ ,
thanks to a reverse order homeomorphism.

\begin{corollary}
\label{decreasinghomeo} Let $g:I\longrightarrow I$ and $f:J\longrightarrow J$
two increasing homeomorphisms defined in the intervals $I=\left[  x_{L}%
,x_{R}\right]  $ and $J=\left[  \overline{x}_{L},\overline{x}_{R}\right]  $.
Suppose that they verify the following conditions:

\begin{enumerate}

\item $g\left(  x\right)  <x$, for all $x\in I=\left(  x_{L},x_{R}\right)  $.

\item $f\left(  x\right)  >x$, for all $x\in J=\left(  \overline{x}%
_{L},\overline{x}_{R}\right)  $.

\item The endpoints of the intervals $I, J$ are the unique fixed points of $g$
and $f$ respectively.


\end{enumerate}

Then, there exists a topological conjugacy between $g$ and $f$, i.e.,
a homeomorphism $h:J\longrightarrow I$, not necessarily unique, such that%
\begin{equation}
g\circ h\left(  x\right)  =h\circ f\left(  x\right)  . \label{conjug4}%
\end{equation}

\end{corollary}

\begin{proof}
First of all note that, in this case, $x_{L}$ and $\overline{x}_{R}$ are
semi-attracting fixed points while $x_{R}$ and $\overline{x}_{L}$ are
semi-repelling fixed points.

The proof is quite similar to the one in Theorem \ref{T1}. The unique
difference is that, in this case, as $g\left(  x\right)  <x$ for all $x\in
I=\left(  x_{L},x_{R}\right)  $, we define a decreasing homeomorphism $h_{0}$
in the domain $D_{0}=\left[  a,f\left(  a\right)  \right]  $
\[
h_{0}:\left[  a,f\left(  a\right)  \right]  \longrightarrow\left[  g\left(
b\right)  ,b\right]  ,
\]
such that $h_{0}\left(  a\right)  =b$ and $h_{0}\left(  f\left(  a\right)
\right)  =g\left(  b\right)  $.

In such a context, the new domains are
\[
D_{n}=\left[  f^{n}\left(  a\right)  ,f^{n+1}\left(  a\right)  \right]  ,
\]%
\[
D_{n}^{\prime}=\left[  g^{n+1}\left(  b\right)  ,g^{n}\left(  b\right)
\right]
\]
and the homeomorphisms are
\[
h_{n}: D_{n}\longrightarrow D_{n}^{\prime}, \quad\text{ }h_{n}\left(
x\right)  =\left(  g^{n}\circ h_{0}\circ f^{-n}\right)  \left(  x\right)  ,
\]
which are decreasing in each of the domains.

The same change has to be considered for the domains
\[
D_{-n}=\left[  f^{-(n+1)}\left(  a\right)  ,f^{-n}\left(  a\right)  \right]
,
\]%
\[
D_{-n}^{\prime}=\left[  g^{-n}\left(  b\right)  ,g^{-(n+1)}\left(  b\right)
\right]
\]
and the homeomorphisms
\[
h_{-n}: D_{-n}\longrightarrow D_{-n}^{\prime} ,\quad\text{ }h_{-n}\left(
x\right)  =\left(  g^{-n}\circ h_{0}\circ f^{n}\right)  \left(  x\right)  ,
\]

With these considerations, similar arguments to those in Theorem \ref{T1} work.
\end{proof}

\vspace*{0.5cm}


\begin{remark}
\label{rem2}As a consequence of the above Corollary \ref{decreasinghomeo}, it
can be deduced that any increasing homeomorphism $f:J\longrightarrow J$ with
two fixed points, $\overline{x}_{L},\overline{x}_{R}$ which are the endpoints
of the interval $J$, is topologically equivalent to its inverse $f^{-1}%
:J\longrightarrow J$. Furthermore, in such a case, $h(x)=\overline{x}%
_{L}+\overline{x}_{R}-x$ is a topological conjugacy between them.
\end{remark}

\subsection{\label{Main}Maps with $n$ fixed points}

The combination of all the results of the previous subsections leaves us in
position to state the following more general result about increasing homeomorphisms.

\begin{theorem}
\label{corollarynfixedpoints} Let $g:I\longrightarrow g(I)$ and
$f:J\longrightarrow f(J)$ two increasing homeomorphisms.
Then, they are topologically conjugated if and only if they have the same
number of fixed points with the same sequence of stabilities or reversed
sequence of stabilities.
\end{theorem}

\begin{proof}
The proof results immediately from a combination and repeated application to
increasing sets of Theorem \ref{XRStable} and Corollaries \ref{XLStable},
\ref{XRuns} and \ref{XLuns} for one fixed point (or intervals lying in the
exterior of the leftmost and rightmost fixed points) and Theorem \ref{T1} and
Corollaries \ref{Cor1} and \ref{decreasinghomeo} for intervals between pairs
of fixed points and observing that two homeomorphisms with zero fixed points
are trivially conjugated.

If the order of the stabilities of the fixed points is reversed from one map
to the other, we use\ reverse homeomorphisms constructed using a reasoning
similar to Remark \ref{rem2} to obtain the conjugacy.
\end{proof}

\begin{example}
Consider three homeomorphisms $\phi$, $\gamma$ and $\zeta$ with three fixed
points each one, $\phi$ and $\gamma$ have the leftmost fixed point attracting,
the middle one semi-repelling to the left and semi-attracting from the right
and the rightmost fixed point repelling, finally $\zeta$ has the leftmost
fixed point repelling, the middle one semi-attracting from the left and
semi-repelling to the right and with the rightmost fixed point attracting.
Accordingly to Theorem \ref{corollarynfixedpoints} the three homeomorphisms
are topologically equivalent.
\end{example}

The general Theorem \ref{corollarynfixedpoints} is enough to prove all the
results of topological conjugacy for bifurcations with eigenvalue $1$ in the
section \ref{Section3}.

\begin{remark}
As a consequence of Theorem \ref{corollarynfixedpoints}, two increasing
homeomorphisms with the same even number of transverse fixed points are
topologically conjugated.
\end{remark}

\begin{remark}
Observe that the Theorem \ref{corollarynfixedpoints}, allows us to deduce that
any increasing homeomorphism with an even number of transverse fixed points is
topologically conjugated to its inverse. Moreover, we can affirm that any
increasing homeomorphism with an odd number of transverse fixed points is not
topologically conjugated to its inverse.
\end{remark}


\subsection{Maps with one fixed point and one 2-periodic orbit}

In this section, we deal with conjugacies between decreasing homeomorphisms
associated to the flip bifurcation. For simplicity, we shall consider that $g$
and $f$ have a fixed point\footnote{In the case of decreasing homeomorphisms,
if there exists a unique fixed point, it must be transverse.} at the origin.
We consider first the topological conjugacy of two homeomorphisms $g$ and $f$
between the two periodic points of each homeomorphism.

\begin{theorem}
\label{T4} Let $g:I\longrightarrow I$ and $f:J\longrightarrow J$ be decreasing
homeomorphisms with domains $I=\left[  x_{L},x_{R}\right]  $ and $J=\left[
\overline{x}_{L},\overline{x}_{R}\right]  $. Suppose that they verify the
following conditions:

\begin{enumerate}
\item $g\left(  x\right)  $ has the repelling fixed point $0$ and an
attracting period two orbit $\left\{  x_{L},x_{R}\right\}  $, with $x_{L}<0$
and $x_{R}>0$, such that $g\left(  x_{L}\right)  =x_{R}$ and $g\left(
x_{R}\right)  =x_{L}$.

\item $f\left(  x\right)  $ has the repelling fixed point $0$ and an
attracting period two orbit $\left\{  \overline{x}_{L},\overline{x}%
_{R}\right\}  $, with $\overline{x}_{L}<0$ and $\overline{x}_{R}>0$, such that
$f\left(  \overline{x}_{L}\right)  =\overline{x}_{R}$ and $f\left(
\overline{x}_{R}\right)  =\overline{x}_{L}$.

\item $g^{2}\left(  x\right)  >x$, for all $x\in\left(  0,x_{R}\right)  $.

\item $f^{2}\left(  x\right)  >x$, for all $x\in\left(  0,\overline{x}%
_{R}\right)  $.
\end{enumerate}

Then, there exists a topological conjugacy between $g$ and $f$, i.e.,
a homeomorphism $h:\left[  \overline{x}_{L},\overline{x}_{R}\right]
\longrightarrow\left[  x_{L},x_{R}\right]  $, not necessarily unique, such
that
\begin{equation}
g\circ h\left(  x\right)  =h\circ f\left(  x\right)  . \label{conjug2}%
\end{equation}

\end{theorem}

\begin{proof}
To prove the result, we fix one image of $h$ at one particular point%
\[
h\left(  a\right)  =b\text{,}%
\]
where $0<a<\overline{x}_{R}$ and $0<b<x_{R}$. We will construct the
homeomorphism $h$ subject to the condition $h\left(  a\right)  =b$.

Now, we consider an arbitrary increasing homeomorphism $h_{0}$ from the domain
$D_{0}=\left[  a,f^{2}\left(  a\right)  \right]  $ onto $D_{0}^{\prime
}=\left[  b,g^{2}\left(  b\right)  \right]  $%
\[
h_{0}:\left[  a,f^{2}\left(  a\right)  \right]  \longrightarrow\left[
b,g^{2}\left(  b\right)  \right]  ,
\]
subject to the conditions $h_{0}\left(  a\right)  =b$ and $h_{0}\left(
f^{2}\left(  a\right)  \right)  =g^{2}\left(  b\right)  $. This homeomorphism
is increasing due to hypotheses $3$ and $4$.

This construction allows us to define the homeomorphism $h_{1}\left(
x\right)  $ in the domain%
\[
D_{1}=\left[  f^{3}\left(  a\right)  ,f\left(  a\right)  \right]
\longrightarrow D_{1}^{\prime}=\left[  g\left(  b\right)  ,g^{3}\left(
b\right)  \right]
\]
and to construct step by step the complete topological conjugacy between $g$
and $f$.

Then, the restriction of $h$ to $D_{0}$ is $h_{0}$ and to $D_{1}$ is $h_{1}$.
Consider the points $x$ in the domain $D_{0}=\left[  a,f^{2}\left(  a\right)
\right]  $, the right side of the conjugacy equation, i.e., $h\circ f\left(
x\right)  $, acts on $D_{0}$, obviously $f\left(  D_{0}\right)  =\left[
f^{3}\left(  a\right)  ,f\left(  a\right)  \right]  =D_{1}$.

In order to compute $h\left(  x\right)  $ when $x\in D_{1}$, we use the left
hand side of the conjugacy equation (\ref{conjug2}) to define $h_{1}\left(
x\right)  $ when $x\in D_{1}$, i.e., the restriction of $h$ to the interval
$D_{1}$ using the rule%
\[
h_{1}\left(  x\right)  =\left(  g\circ h_{0}\circ f^{-1}\right)  \left(
x\right)  \text{, }x\in D_{1}.
\]
Once again, this homeomorphism is increasing due to the fact that $f^{-1}$ and
$g$ are decreasing and $h_{0}$ is increasing.

Defining the sequence of intervals%
\begin{align*}
D_{2n}  &  =\left[  f^{2n}\left(  a\right)  ,f^{2n+2}\left(  a\right)
\right]  =f^{2n}\left(  D_{0}\right)  ,\\
D_{2n+1}  &  =\left[  f^{2n+3}\left(  a\right)  ,f^{2n+1}\left(  a\right)
\right]  =f^{2n+1}\left(  D_{0}\right)  ,
\end{align*}
it is possible to extend the definition of successive restrictions of $h$ to
the intervals $D_{n}$, using the same procedure. Thus, for $n=2$, we get%
\[
h_{2}\left(  x\right)  =\left(  g\circ h_{1}\circ f^{-1}\right)  \left(
x\right)  \text{, }x\in D_{2}%
\]
or%
\[
h_{2}\left(  x\right)  =\left(  g^{2}\circ h_{0}\circ f^{-2}\right)  \left(
x\right)  \text{, }x\in D_{2}\text{.}%
\]
In general%
\[
h_{n}\left(  x\right)  =\left(  g^{n}\circ h_{0}\circ f^{-n}\right)  \left(
x\right)  \text{, }x\in D_{n}\text{,}%
\]
which is increasing.

Let us see that this definition works. Let $x\in D_{j}$,
\begin{align*}
g\circ h_{j}\left(  x\right)   &  =h_{j+1}\circ f\left(  x\right) \\
g\circ(g^{j}\circ h_{0}\circ f^{-j})\left(  x\right)   &  =\left(
g^{j+1}\circ h_{0}\circ f^{-j-1}\right)  \circ f\left(  x\right) \\
g^{j+1}\circ h_{0}\circ f^{-j}\left(  x\right)   &  =g^{j+1}\circ h_{0}\circ
f^{-j}\left(  x\right)  \text{.}%
\end{align*}

Following a similar reasoning as in theorems before, one can easily check that
the union of the intervals $D_{2n}=\left[  f^{2n}\left(  a\right)
,f^{2n+2}\left(  a\right)  \right]  $ is%
\[%
{\displaystyle\bigcup\limits_{n=0}^{\infty}}
D_{2n}=\left[  a,\overline{x}_{R}\right)  ,
\]
because $\overline{x}_{R}$ is the unique positive fixed point of $f^{2}$ and
the initial condition $a$ is positive. Analogously, the union of the intervals
$D_{2n+1}=\left[  f^{2n+3}\left(  a\right)  ,f^{2n+1}\left(  a\right)
\right]  $ is%
\[%
{\displaystyle\bigcup\limits_{n=0}^{\infty}}
D_{2n+1}=\left(  \overline{x}_{L},f\left(  a\right)  \right]  ,
\]
since $\overline{x}_{L}$ is an attracting fixed point of $f^{2}$ and the
initial condition $f\left(  a\right)  $ is negative.

On the other hand, the piecewise function $h_{\rightarrow}^{\circ}$ defined by%
\[
h_{\rightarrow}^{\circ}\left(  x\right)  =\left\{
\begin{array}
[c]{cccc}%
\cdots &  &  & \bigskip\\
h_{2n+1}\left(  x\right)  : & \left[  f^{2n+3}\left(  a\right)  ,f^{2n+1}%
\left(  a\right)  \right]  & \rightarrow & \left[  g^{2n+3}\left(  b\right)
,g^{2n+1}\left(  b\right)  \right]  \bigskip\\
\cdots &  &  & \bigskip\\
h_{3}\left(  x\right)  : & \left[  f^{5}\left(  a\right)  ,f^{3}\left(
a\right)  \right]  & \rightarrow & \left[  g^{5}\left(  b\right)
,g^{3}\left(  b\right)  \right]  \bigskip\\
h_{1}\left(  x\right)  : & \left[  f^{3}\left(  a\right)  ,f\left(  a\right)
\right]  & \rightarrow & \left[  g^{3}\left(  b\right)  ,g\left(  b\right)
\right]  \bigskip\\
h_{0}\left(  x\right)  : & \left[  a,f^{2}\left(  a\right)  \right]  &
\rightarrow & \left[  b,g^{2}\left(  b\right)  \right]  \bigskip\\
h_{2}\left(  x\right)  : & \left[  f^{2}\left(  a\right)  ,f^{4}\left(
a\right)  \right]  & \rightarrow & \left[  g^{2}\left(  b\right)
,g^{4}\left(  b\right)  \right]  \bigskip\\
\cdots &  &  & \bigskip\\
h_{2n}\left(  x\right)  : & \left[  f^{2n}\left(  a\right)  ,f^{2n+2}\left(
a\right)  \right]  & \rightarrow & \left[  g^{2n}\left(  b\right)
,g^{2n+2}\left(  b\right)  \right]  \bigskip\\
\cdots &  &  & ,
\end{array}
\right.
\]
is continuous in each interval and agrees at each extreme point of two
successive intervals. Moreover, $h_{j}=g^{j}\circ h_{0}\circ f^{-j}$ is the
composition of two decreasing functions with an increasing function when $j$
is odd and three increasing functions when $j$ is even. Thus, it is increasing
in $D_{j}$. We can prolong again $h_{\rightarrow}^{\circ}$ by continuity to
$h_{\rightarrow}$ in the closed interval making $h\left(  \overline{x}%
_{L}\right)  =x_{L}$ and $h\left(  \overline{x}_{R}\right)  =x_{R}$. This
implies that $h_{\rightarrow}$ is an increasing homeomorphism in $\left[
\overline{x}_{L},f(a)\right]  $ and in $\left[  a,\overline{x}_{R}\right]  $
.

The second part of the proof concerns the backward process. Using a similar
reasoning, we construct the topological conjugacy to the pre-images of $D_{0}$.

We remember that the restriction of $h$ to $D_{0}$ is $h_{0}$. Consider the
points $x$ in $D_{0}=\left[  a,f^{2}\left(  a\right)  \right]  $, then we can
define the restrictions $h_{-n}$ of $h$ to the intervals%
\begin{align*}
D_{-2n}  &  =\left[  f^{-2n}\left(  a\right)  ,f^{-2n+2}\left(  a\right)
\right]  =f^{-2n}\left(  D_{0}\right)  ,\\
D_{-2n+1}  &  =\left[  f^{-2n+3}\left(  a\right)  ,f^{-2n+1}\left(  a\right)
\right]  =f^{-2n+1}\left(  D_{0}\right)  ,
\end{align*}
for $n\geq1$.

Using the same type of arguments of the forward construction, we find that%
\[
h_{-n}\left(  x\right)  =g^{-n}\circ h_{0}\circ f^{n}\left(  x\right)  \text{,
}x\in D_{-n}\text{.}%
\]
As in Theorem \ref{T1} this type of construction works. Let $x\in D_{-j}$ with
$-n<-j\leq0$, then%
\begin{align*}
g\circ h_{-j}\left(  x\right)   &  =h_{-j+1}\circ f\left(  x\right) \\
g\circ(g^{-j}\circ h_{0}\circ f^{j})\left(  x\right)   &  =\left(
g^{-j+1}\circ h_{0}\circ f^{j-1}\right)  \circ f\left(  x\right) \\
g^{-j+1}\circ h_{0}\circ f^{j}\left(  x\right)   &  =g^{-j+1}\circ h_{0}\circ
f^{j}\left(  x\right)  \text{.}%
\end{align*}
The rest of the backward process is just the same as in the forward case,
since now $0$ is an attracting fixed point for both $f^{-1}$ and $g^{-1}$. We
define $h_{\leftarrow}^{\circ}$ piecewise%
\[
h_{\leftarrow}^{\circ}\left(  x\right)  =\left\{
\begin{array}
[c]{cccc}%
h_{-1}\left(  x\right)  : & \left[  f\left(  a\right)  ,f^{-1}\left(
a\right)  \right]  & \rightarrow & \left[  g\left(  b\right)  ,g^{-1}\left(
b\right)  \right]  \bigskip\\
h_{-3}\left(  x\right)  : & \left[  f^{-1}\left(  a\right)  ,f^{-3}\left(
a\right)  \right]  & \rightarrow & \left[  g^{-1}\left(  b\right)
,g^{-3}\left(  b\right)  \right]  \bigskip\\
\cdots &  &  & \bigskip\\
h_{_{-2n+1}}\left(  x\right)  : & \left[  f^{^{-2n+3}}\left(  a\right)
,f^{^{-2n+1}}\left(  a\right)  \right]  & \rightarrow & \left[  g^{^{-2n+3}%
}\left(  b\right)  ,g^{^{-2n+1}}\left(  b\right)  \right]  \bigskip\\
\cdots &  &  & \bigskip\\
h_{_{-2n}}\left(  x\right)  : & \left[  f^{^{-2n}}\left(  a\right)
,f^{^{-2n+2}}\left(  a\right)  \right]  & \rightarrow & \left[  g^{^{-2n}%
}\left(  b\right)  ,g^{^{-2n+2}}\left(  b\right)  \right]  \bigskip\\
\cdots &  &  & \bigskip\\
h_{-4}\left(  x\right)  : & \left[  f^{-4}\left(  a\right)  ,f^{-2}\left(
a\right)  \right]  & \rightarrow & \left[  g^{-4}\left(  b\right)
,g^{-2}\left(  b\right)  \right]  \bigskip\\
h_{-2}\left(  x\right)  : & \left[  f^{-2}\left(  a\right)  ,a\right]  &
\rightarrow & \left[  g^{-2}\left(  b\right)  ,b\right]
\end{array}
\right.
\]
Noticing that $\lim_{n\rightarrow\infty}f^{-n}\left(  a\right)  =0$ and
$\lim_{n\rightarrow\infty}g^{-n}\left(  b\right)  =0$, we prolong
$h_{\leftarrow}^{\circ}$ to a homeomorphism $h_{\leftarrow}$\ making
$h_{\leftarrow}\left(  0\right)  =0$ as before.

Now, joining $h_{\leftarrow}$ and $h_{\rightarrow}$ at their corresponding
intervals, we have just constructed $h$ at the whole interval $J=\left[
\overline{x}_{L},\overline{x}_{R}\right]  $, as desired.
\end{proof}

\vspace*{0.5cm}

Now, we have to consider what happens outside the intervals that contain the
periodic points of both $g$ and $f$.

\begin{theorem}
\label{T5} Let $g:I\longrightarrow I$ and $f:J\longrightarrow J$ be decreasing
homeomorphisms with domains $I=\left[  g\left(  b\right)  ,b\right]  $ and
$J=\left[  f\left(  a\right)  ,a\right]  $. Suppose that they verify the conditions:

\begin{enumerate}
\item $g\left(  x\right)  $ has the repelling fixed point $0$ and an
attracting period two orbit $\left\{  x_{L},x_{R}\right\}  $ with $x_{L}<0$
and $x_{R}>0$, such that $g\left(  x_{L}\right)  =x_{R}$ and $g\left(
x_{R}\right)  =x_{L}$. Moreover $b>x_{R}$.

\item $f\left(  x\right)  $ has the repelling fixed point $0$ and an
attracting period two orbit $\left\{  \overline{x}_{L},\overline{x}%
_{R}\right\}  $ with $\overline{x}_{L}<0$ and $\overline{x}_{R}>0$, such that
$f\left(  \overline{x}_{L}\right)  =\overline{x}_{R}$ and $f\left(
\overline{x}_{R}\right)  =\overline{x}_{L}$. Moreover $a>\overline{x}_{R}$.

\item $g^{2}\left(  x\right)  <x$, for all $x\in\left(  x_{R},b\right]  $.

\item $f^{2}\left(  x\right)  <x$, for all $x\in\left(  \overline{x}%
_{R},a\right]  $.
\end{enumerate}

Then, there exists a topological conjugacy between $g$ and $f$, i.e.,
a homeomorphism
\[
h:\left[  f\left(  a\right)  ,\overline{x}_{L}\right]  \cup\left[
\overline{x}_{R},a\right]  \longrightarrow\left[  g\left(  b\right)
,x_{L}\right]  \cup\left[  x_{R},b\right]
\]
not necessarily unique, such that%
\[
g\circ h\left(  x\right)  =h\circ f\left(  x\right)  \text{,}%
\]

\end{theorem}

\begin{proof}
The proof is similar to the proof of Theorem \ref{XRStable}. Consider $a$, the
rightmost point of $\left[  \overline{x}_{R},a\right]  $. Since $f^{2}\left(
x\right)  <x$, we have $f^{2}\left(  a\right)  <a$. Consider the interval
$D_{0}=\left[  f^{2}\left(  a\right)  ,a\right]  $ and any homeomorphism
$h_{0}\left(  x\right)  $ subject to the conditions $h_{0}\left(  a\right)
=b$ and $h_{0}\left(  f^{2}\left(  a\right)  \right)  =g^{2}\left(  b\right)
$. Due to conditions $3$ and $4$ above, $h_{0}$ is increasing since
$f^{2}\left(  a\right)  <a$ and $g^{2}\left(  b\right)  <b$.

Consider now domains
\begin{align*}
D_{2n}  &  =\left[  f^{2n+2}\left(  a\right)  ,f^{2n}\left(  a\right)
\right]  ,\\
D_{2n+1}  &  =\left[  f^{2n+1}\left(  a\right)  ,f^{2n+3}\left(  a\right)
\right]
\end{align*}
and%
\begin{align*}
D_{2n}^{\prime}  &  =\left[  g^{2n+2}\left(  b\right)  ,g^{2n}\left(
b\right)  \right] \\
D_{2n+1}^{\prime}  &  =\left[  g^{2n+1}\left(  b\right)  ,g^{2n+3}\left(
b\right)  \right]
\end{align*}
and the homeomorphisms
\[
h_{n}: D_{n}\longrightarrow D_{n}^{\prime},\quad\text{ } h_{n}\left(
x\right)  =\left(  g^{n}\circ h_{0}\circ f^{-n}\right)  \left(  x\right)  ,
\]
which are increasing in each domain, since they are the composition of
increasing functions (when $n$ is even) or two decreasing functions and an
increasing function (when $n$ is odd).

At this point, we can consider the piecewise function $h^{\circ}$ defined by%
\[
h^{\circ}\left(  x\right)  =\left\{
\begin{array}
[c]{cccc}%
h_{1}\left(  x\right)  : & \left[  f\left(  a\right)  ,f^{3}\left(  a\right)
\right]  & \rightarrow & \left[  g\left(  b\right)  ,g^{3}\left(  b\right)
\right]  \bigskip\\
h_{3}\left(  x\right)  : & \left[  f^{3}\left(  a\right)  ,f^{5}\left(
a\right)  \right]  & \rightarrow & \left[  g^{3}\left(  b\right)
,g^{5}\left(  b\right)  \right]  \bigskip\\
\cdots &  &  & \bigskip\\
h_{_{2n+1}}\left(  x\right)  : & \left[  f^{^{2n+1}}\left(  a\right)
,f^{^{2n+3}}\left(  a\right)  \right]  & \rightarrow & \left[  g^{^{2n+1}%
}\left(  b\right)  ,g^{^{2n+3}}\left(  b\right)  \right]  \bigskip\\
\cdots &  &  & \bigskip\\
h_{_{2n}}\left(  x\right)  : & \left[  f^{^{2n+2}}\left(  a\right)  ,f^{^{2n}%
}\left(  a\right)  \right]  & \rightarrow & \left[  g^{^{2n+2}}\left(
b\right)  ,g^{^{2n}}\left(  b\right)  \right]  \bigskip\\
\cdots &  &  & \bigskip\\
h_{2}\left(  x\right)  : & \left[  f^{4}\left(  a\right)  ,f^{2}\left(
a\right)  \right]  & \rightarrow & \left[  g^{4}\left(  b\right)
,g^{2}\left(  b\right)  \right]  \bigskip\\
h_{0}\left(  x\right)  : & \left[  f^{2}\left(  a\right)  ,a\right]  &
\rightarrow & \left[  g^{2}\left(  b\right)  ,b\right]  ,
\end{array}
\right.
\]
which is continuous. Since
\[
\lim_{n\rightarrow\infty}f^{2n}\left(  a\right)  =\overline{x}_{R}\text{ and
}\lim_{n\rightarrow\infty}g^{2n}\left(  b\right)  =x_{R},
\]
we have
\[
\lim_{x\rightarrow\overline{x}_{R}}h^{\circ}\left(  x\right)  =x_{R},
\]
and since
\[
\lim_{n\rightarrow\infty}f^{2n+1}\left(  a\right)  =\overline{x}_{L}\text{ and
}\lim_{n\rightarrow\infty}g^{2n+1}\left(  b\right)  =x_{L},
\]
we have
\[
\lim_{x\rightarrow\overline{x}_{L}}h^{\circ}\left(  x\right)  =x_{L}.
\]
Prolonging $h^{\circ}$ to $h$ at $\overline{x}_{L}$ and $\overline{x}_{R}$ as
before, we get the topological conjugacy with the desired properties.
\end{proof}

\vspace*{0.5cm}


Now, we are going to consider the case previous to the appearance of the
period two orbit in the bifurcation, that is, with both periodic points
collapsed to the origin.

\begin{theorem}
\label{T6} Let $g:I\longrightarrow I$ and $f:J\longrightarrow J$ be decreasing
homeomorphisms with domains $I=\left[  f\left(  a\right)  ,a\right]  $ and
$J=\left[  g\left(  b\right)  ,b\right]  $. Suppose that they verify the conditions:

\begin{enumerate}
\item $f\left(  x\right)  $ and $g\left(  x\right)  $ have the attracting
fixed point $0$.

\item Consider $b>0$, we have $g^{2}\left(  x\right)  <x$, for all
$x\in\left(  0,b\right]  $.

\item Consider $a>0$, $f^{2}\left(  x\right)  <x$, for all $x\in\left(
0,a\right]  $.
\end{enumerate}

Then, there exists a topological conjugacy between $g$ and $f$, i.e.,
a homeomorphism $h:\left[  f\left(  a\right)  ,a\right]  \longrightarrow
\left[  g\left(  b\right)  ,b\right]  $, not necessarily unique, such that%
\[
g\circ h\left(  x\right)  =h\circ f\left(  x\right)  \text{.}%
\]

\end{theorem}

\begin{proof}
The proof is similar to the one of Theorem \ref{T5}, but with $\overline
{x}_{L}$ and $\overline{x}_{R}$ collapsed to the origin.

\end{proof}

\vspace*{0.5cm}

\begin{remark}
\label{rem}Similar results to Theorems \ref{T4}, \ref{T5} and \ref{T6} can be
obtained for decreasing homeomorphisms $g\left(  x\right)  $ and $f\left(
x\right)  $ that have an attracting fixed point at $0$ and a repelling period
two orbit or only a repelling fixed point.
\end{remark}


\section{\label{Section3}Solution for the conjugacy problem for the fold,
transcritical, pitchfork and flip bifurcations normal forms}

The results in the previous sections allows us to establish local topological
conjugacies between any two homeomorphisms with the same number of fixed
points (or period-2 points) and with the same stability. These more general
results can be used, in particular, to give a complete proof for the problem
of finding a homeomorphism providing the topological equivalence between any
family verifying specific bifurcation conditions and the corresponding
simplest (truncated) normal form, which remained unpublished until now, as
said in \cite{Kuz95}.

Actually, they can be used in a more general context of appearance of these
bifurcations under generalized conditions (see \cite{BaVa99}), giving
topological equivalences between any family satisfying specific bifurcation
conditions of any order and the simplest (truncated) normal form. In
particular, normal forms of any order for the same bifurcation (see
\cite{BaVa00}) are topologically equivalent. Thus, theorems below give a
solution for this classical problem in bifurcation theory even in a more
general way than originally posed three decades ago in \cite{Arn83}.

Specifically, here we show how to construct homeomorphisms which allows us to
find (local) topological equivalences among one-parameter families undergoing
the same kind of bifurcation.

We have a local topological equivalence of families $f_{\alpha}$ and
$g_{\beta}$ depending on parameters $\alpha$ and $\beta$ whenever
\cite{AAIS94,Kuz95}:

\begin{enumerate}
\item there exists a homeomorphism of the parameter space $p:$ $%
\mathbb{R}
$ $\longrightarrow%
\mathbb{R}
$, $\beta=p\left(  \alpha\right)  $;

\item there is a parameter-dependent homeomorphism of the phase space
$h_{\alpha}$ : $J\longrightarrow I$, $y=h_{\alpha}\left(  x\right)  $, mapping
orbits of the system $f_{\alpha}$ onto orbits of $g_{\beta}$, i.e.,
$h_{\alpha}\circ f_{\alpha}=g_{\beta}\circ h_{\beta}$, at parameter values
$\beta=p\left(  \alpha\right)  $.
\end{enumerate}

As can be done, we consider the same denomination for the parameter $\mu$ for
the two maps $f_{\mu}$ and $g_{\mu}$ which we want to conjugate. Thus, the
homeomorphism $p$ is the identity, $\alpha=\beta=\mu$ and $h_{\mu}\circ
f_{\mu}=g_{\mu}\circ h_{\mu}$. One does not require the homeomorphism $h_{\mu
}$ to depend continuously on the parameter $\mu$, meaning that the conjugacy
of the families $f_{\mu}$ and $g_{\mu}$ is a weak (or fiber) equivalence (see
\cite{AAIS94} or \cite{Kuz95}).


\begin{theorem}
(Fold Normal Form)\label{Fold Normal Form} Let $f:\mathbb{R}\times
\mathbb{R}\rightarrow\mathbb{R}$ be a one-parameter family of maps verifying
the generalized nondegeneracy conditions of a fold bifurcation \cite{BaVa99}.
Then, $f_{\mu}$ is (locally) topologically equivalent near the origin to one
of the following normal forms
\begin{equation}
g(x,\mu)=x\pm\mu\pm x^{2} \label{FoldSimplestNormalForm}%
\end{equation}

\end{theorem}

\begin{proof}
It is sufficient to observe that on one side of $\mu=0$, any element of the
family $f_{\mu}$ has two transverse fixed points, one unstable and one stable
and $g_{\mu}$ has exactly the same number of transverse fixed points one
stable and the other unstable, therefore $f_{\mu}$ and $g_{\mu}$ satisfy the
conditions of Theorem \ref{corollarynfixedpoints}; while in the other side of
$\mu=0$ any system $f_{\mu}$ has no fixed points also satisfying Theorem
\ref{corollarynfixedpoints}.

On the other hand, the maps $f_{0}$ and $g_{0}$ have a unique mixed-stability
fixed point with the same stability, i.e., semi-attracting from one side and
semi-repelling to the other side or vice-versa. Consequently, both maps verify
again Theorem \ref{corollarynfixedpoints} at the bifurcation point.
\end{proof}

\begin{example}
\textrm{The Theorem \ref{Fold Normal Form} allows us to know that every family
of any of the forms
\[
x\pm\mu\pm x^{2}+\text{h.o.t.}%
\]
is locally topologically equivalent to one of the forms in
(\ref{FoldSimplestNormalForm}), so giving a complete proof for the classical
bifurcation problem exactly as posed in \cite{Arn83} or \cite{Kuz95}.
Moreover, it also allows us to know that any higher order normal form of the
fold bifurcation
\[
x\pm\mu\pm x^{2n},\quad n\in\mathbb{N}%
\]
is topologically equivalent to the simplest one given in
(\ref{FoldSimplestNormalForm}) and, even more generally, that any family of
the form
\[
x\pm\mu\pm x^{2n}+h.o.t.,\quad n\in\mathbb{N}%
\]
is also topologically equivalent to one the simplest ones given in
(\ref{FoldSimplestNormalForm}). }
\end{example}

\begin{theorem}
\label{Transcritical}(Transcritical Normal Form) Let $f:\mathbb{R}%
\times\mathbb{R}\rightarrow\mathbb{R}$ be a one-parameter family of maps
verifying the generalized nondegeneracy conditions of a transcritical
bifurcation \cite{BaVa99}.
Then, $f_{\mu}$ is (locally) topologically equivalent near the origin to one
of the following normal forms
\begin{equation}
g(x,\mu)=x\pm\mu x\pm x^{2} \label{transcriticaltruncated}%
\end{equation}

\end{theorem}

\begin{proof}
It is sufficient to observe that on both sides of $\mu=0$, any system $f_{\mu
}$ has two transverse fixed points, which alternate their stabilities and
$g_{\mu}$ has exactly the same number of fixed points one stable and the other
unstable. Hence, when $\mu\not =0$, $f_{\mu}$ and $g_{\mu}$ satisfy the
conditions of Theorem \ref{corollarynfixedpoints}, being topologically conjugated.

On the other hand, for $f_{0}$ and $g_{0}$ there exists for each map a unique
mixed-stability fixed point. So, both maps verify again Theorem
\ref{corollarynfixedpoints}.
\end{proof}

\begin{example}
\textrm{The Theorem \ref{Transcritical} allows us to know that every family of
any the forms
\[
x\pm\mu x\pm x^{2n}+\text{h.o.t.},\quad n\in\mathbb{N}%
\]
is locally topologically equivalent to one of the forms in
(\ref{transcriticaltruncated}). }
\end{example}

\begin{theorem}
\label{Pitchfork}(Pitchfork Normal Form) Let $f:\mathbb{R}\times
\mathbb{R}\rightarrow\mathbb{R}$ be a one-parameter family of maps verifying
the generalized nondegeneracy conditions of a pitchfork bifurcation
\cite{BaVa99}.
Then, $f_{\mu}$ is (locally) topologically equivalent near the origin to one
of the following normal forms
\begin{equation}
g(x,\mu)=x\pm\mu x\pm x^{3} \label{pitchforktruncated}%
\end{equation}

\end{theorem}

\begin{proof}
It is sufficient to observe that on one side of $\mu=0$, any map in the family
$f_{\mu}$ has three transverse fixed points, one middle point $x_{M}$ and two
exterior fixed points $x_{L}$ and $x_{R}$ which have the opposite stability of
$x_{M}$. Then, $f_{\mu}$ and $g_{\mu}$ share the same number of fixed points
with the same stability and we apply Theorem \ref{corollarynfixedpoints}.
While, on the other side of $\mu=0$, any map in the family $f_{\mu}$ has only
one transverse fixed point which is $x_{M}$ with the same stability of $0$ for
$g_{\mu}$. Therefore, both maps verify Theorem \ref{corollarynfixedpoints}.
The same happens for $f_{0}$ and $g_{0}$. For each one of those functions
there exists a unique transverse fixed point with the same stability.
\end{proof}

\begin{example}
\textrm{The Theorem \ref{Pitchfork} allows us to know that every family of any
of the forms
\[
x\pm\mu x\pm x^{2n+1}+\text{h.o.t.},\quad n\in\mathbb{N}%
\]
is locally topologically equivalent to one of the forms in
(\ref{pitchforktruncated}). }
\end{example}

\begin{theorem}
\label{flip} (Flip Normal Form) Let $f:\mathbb{R}\times\mathbb{R}%
\rightarrow\mathbb{R}$ be a one-parameter family of maps verifying the
generalized nondegeneracy conditions of a flip bifurcation \cite{BaVa99}.
Then, $f_{\mu}$ is (locally) topologically equivalent near the origin to one
of the following normal forms
\begin{equation}
g(x,\mu)=-x\pm\mu x\pm x^{3} \label{fliptruncated}%
\end{equation}

\end{theorem}

\begin{proof}
It is sufficient to observe that on one side of $\mu=0$, any map of the family
$f_{\mu}$ has the fixed point $0$ and two $2$-periodic points with the
opposite stability of $0$. Then, $f_{\mu}$ and $g_{\mu}$ satisfy the
conditions of Theorems \ref{T4}, \ref{T5} and Remark \ref{rem}; while on the
other side of $\mu=0$ any system $f_{\mu}$ has only a fixed point which is $0$
and both maps verify Theorems \ref{T6} and Remark \ref{rem}. The same occurs
for $f_{0}$ and $g_{0}$. For each one of these maps there exists a unique
fixed point with the same stability.
\end{proof}

\begin{example}
\textrm{The Theorem \ref{flip} allows us to know that every family of any the
forms
\[
-x\pm\mu x\pm x^{2n+1}+\text{h.o.t.},\quad n\in\mathbb{N}%
\]
is locally topologically equivalent to one of the forms in
(\ref{fliptruncated}). }
\end{example}

\begin{remark}
As is evident by the proofs, these (local) topological equivalences cover all
the points that are relevant for the corresponding local bifurcations.
\end{remark}

\begin{remark}
It is worthwhile to observe that topological equivalences do not depend on the
order of differentiability of the family $f$, which only must verify the
corresponding bifurcation conditions.
\end{remark}

\section{\label{Section4}Conclusions and future research directions}

This paper shows how to find conjugacies which demonstrate the topological
equivalence between any two increasing homeomorphisms with the same number of
fixed points and with the same sequence of semi-stabilities and, for two
decreasing homeomorphisms, both with one fixed point with no periodic orbits
or both with a fixed point and one $2$-period point.

As a consequence, these general results are applied to give a solution for a
classical bifurcation problem which was set out several decades ago, but
remained unpublished until now. Such a classical problem consists in finding
topological conjugacies between families verifying specific bifurcation
conditions and the simplest (truncated) normal form corresponding to the
bifurcation. Actually, the problem is solved in more general way than
originally posed, since the results work even for generalized bifurcation
conditions independently of the order of differentiability of the families considered.

The topological conjugacies for families are \emph{fiber} or \emph{weak} in
the sense stated in \cite{AAIS94}. We leave as a future research work to
demonstrate if it is possible to obtain the continuity of the homeomorphisms
with respect to the parameter appearing in the families.

\section*{Acknowledgements}

Henrique Oliveira was funded by FCT-Portugal through the research project UID/MAT/04459/2013,.

Francisco Balibrea and Jose C. Valverde thank the Ministry of Sciences and
Innovation of Spain for their support for this work through the grant MTM2011-23221.

The final publication is available at Springer via

http://dx.doi.org/10.1007/s00332-016-9347-0.

\newpage


\begin{thebibliography}{99}                                                                                               %


\bibitem {Arn83}
V.I. Arnold, \emph{Geometrical Methods in the Theory of Ordinary Differential
Equations}, Grundleheren der mathematischen Wissenschaften, 250 (A Series of
Comprehensive Studies in Mathematics). Springer-Verlag: New York, Heidelberg,
Berlin, 1983.


\bibitem {AAIS94}Arnol'd,V., Afraimovich,V., Il'yashenko, Y. \& Shil'nikov,L.,
Bifurcation theory, in V. Arnold, ed., `Dynamical Systems V. Encyclopaedia of
Mathematical Sciences', Springer-Verlag, New York, 1994.

\bibitem {BaVa99}F. Balibrea and J.C. Valverde, Bifurcations Under
Non-degenerated Conditions of Higher Degree and a New Simple Proof of the
Hopf-Neimark-Sacker Bifurcation Theorem, \textit{J. Math. Anal. Appl.}
\textbf{237} (1999), 93-105.

\bibitem {BaVa00}F. Balibrea and J.C. Valverde, Topological Normal Forms of
Higher Degree for the Simplest Bifurcations. \emph{Applied General Topology}
\textbf{2} (2000), 155-164.

\bibitem {Car81}J. Carr, \textit{Applications of Center Manifold Theory}.
Appl. Math. Sci., 35. Springer-Verlag: New York, Heidelberg, Berlin, 1981.

\bibitem {Elaydi}S. N. Elaydi, \textit{An Introduction to difference
equations}. Third Edition. Undergraduate Texts in Mathematics,
Springer-Verlag: New York, Heidelberg, Berlin, 2005. ISBN 0-387-23059-9.

\bibitem {Guc77}J. Guckenheimer, On the bifurcation of maps of the interval,
\emph{Inventiones mathematicae} \textbf{39} (2) (1977), 165-178.

\bibitem {GuHo83}J. Guckenheimer y P. Holmes, \textit{Nonlinear Oscillations,
Dynamical Systems and Bifurcations of Vector Fields}. Appl. Math. Sci., 42.
Springer-Verlag: New York, Heidelberg, Berlin, 1983.

\bibitem {Kuz95}Y.A. Kuznetsov, \textit{Elements of Applied Bifurcation
Theory}. 2nd Edition. Appl. Math. Sci., 112. Springer-Verlag: New York, 1998.

\bibitem {Melo}W. de Melo y S. Van Strien [1993]. \textit{One Dimensional
Dynamics}. Ergebnisse der Mathematik und ihrer Grenzgebiete, 25
Springer-Verlag: Berlin.

\bibitem {Poi29}
H. Poincar\'{e}, \textit{Sur les Propi\'et\'es des Fonctions D\'efinies par
les \'Equations aux Diff\'erences Partielles}, Oeuvres, Gauthier-Villars:
Paris (1929)

\bibitem {Wig90}S. Wiggins \textit{Introduction to Applied Nonlinear Systems
and Chaos}. Texts Appl. Math., 2. Springer-Verlag: New York, 1990.
\end{thebibliography}
\end{document}